\documentclass{article}

\usepackage{amssymb}
\usepackage{amsmath}
\usepackage{amsthm}

\newtheorem{theorem}{Theorem}
\newtheorem{lemma}{Lemma}[section]
\newtheorem{propo}[lemma]{Proposition}
\newtheorem{claim}[lemma]{Claim}
\newtheorem{corol}[lemma]{Corollary}
\newtheorem{defin}[lemma]{Definition}

\newcommand{\N}{\mathbb{N}}
\newcommand{\R}{\mathbb{R}}

\newcommand{\Lattice}{{\mathbb N^m}}
\newcommand{\Space}{{\mathbb R^m}}
\newcommand{\norm}[1]{{\|#1\|}}
\newcommand{\CF}{{\mathcal F}}
\newcommand{\CT}{{\mathcal T}}

\begin{document}

\title{Sofic representations of amenable groups.
\footnote{AMS
Subject Classification: Primary 20F65,
\, Research sponsored by OTKA Grants No. 69062 and NK 78439 }}

\author{G\'abor Elek, Endre Szab\'o}

\sloppy

\maketitle

\begin{abstract} Using probabilistic methods, Collins and Dykema proved that
the free product of two sofic groups amalgamated over a monotileably amenable
subgroup is sofic as well. 
We show that the restriction is unnecessary; the free product of two sofic
groups amalgamated over an arbitrary amenable
subgroup is sofic. We also prove a group theoretical analogue of a 
result of Kenley Jung.
A finitely generated group is amenable if and only if it has only one sofic
representation up to conjugacy equivalence. 
\end{abstract}
\section{Introduction}
First, let us recall the notion of sofic groups (see \cite{Pes} for a survey)
. Let $\Gamma$ be a finitely
generated group with a symmetric generating system $S$. Let $G$ be a finite
graph, such that each directed edge of $G$ is labeled by
an element of $S$. We say that $G$ is an $r$-approximation of the Cayley-graph
of $\Gamma$, $Cay(\Gamma,S)$ if
there exists a subset $W\subseteq V(G)$ such that
\begin{itemize}
\item $|W|> (1-\frac{1}{r})| V(G)|$ and
\item if $p\in W$, then the $r$-neighborhood of $p$ is rooted isomorphic to
  the $r$-neighborhood of a vertex of the Cayley-graph of $\Gamma$ (as edge
  labeled graphs).
\end{itemize}
The group $\Gamma$ is sofic if for any $r\geq 1$ there exist
$r$-approximations of $Cay(\Gamma,S)$ by finite graphs.
We call a not necessarily finitely generated group sofic if all of its
finitely generated subgroups are sofic. Amenable and residually finite groups
are sofic. Free products, direct products, inverse and direct limits of sofic
groups are all sofic groups as well. Amenable extensions of sofic groups are
also sofic. It is important to note that no non-sofic group is known at this
moment. Several conjectures about groups, such as the Gottschalk's Conjecture,
 Kaplansky's Direct Finiteness Conjecture, the
Determinant Conjecture and the Connes Embedding Conjecture hold for sofic 
groups.
\noindent
Using probabilistic arguments, Collins and Dykema \cite{CD} proved that the 
free
products of sofic groups amalgamated over monotileably amenable groups are sofic.
We show that the restriction is not necessary.
\begin{theorem} \label{main} Free products of
sofic groups amalgamated over arbitrary amenable groups are sofic. 
\end{theorem}
In \cite{KY}, Kenley Jung gave a characterization of amenable von
Neumann algebras among ``embeddable'' von Neumann algebras. He proved that
an embeddable von Neumann algebra $M$ is amenable if and only if any two
embeddings of $M$ into $R^\omega$ are conjugate by a unitary element. 
Before stating
our result, let us briefly recall an equivalent definition of soficity. 
Let $\omega$
be a non-principal ultrafilter on the natural numbers and
$\lim_\omega:l^\infty(\N) \to \R$ be the associated ultralimit.
Consider the product of finite symmetric groups $\prod^\infty_{n=1} S_n$. If
$g\in S_n$ then let the rank of $g$ defined as
$rk(g):=1-\frac{fix(\gamma)}{n}$, where $fix(g)$ is the number of elements in
the set $\{1,2,\dots,n\}$ fixed by the permutation $g$. Let $H\subset
\prod^\infty_{n=1} S_n$ be defined as follows.
$$H:=\{ (g_n)^\infty_{n=1}\,\mid\, \lim_\omega rk(g)=0\,\}$$
Then $H$ is a normal subgroup of $\prod^\infty_{n=1} S_n$ and
$\Sigma=\prod^\infty_{n=1} S_n/H$ is a simple group with a rank function
$$rk(\underline{g})=\lim_\omega rk(q_n)\,.$$
We call an embedding $\Gamma\to\Sigma$ faithful if $rk(\gamma)=1$ for any
$1\neq\gamma\in\Gamma$. It is easy to see that sofic approximations lead to
faithful embeddings of countable sofic groups. Conversely, any embeddable group
is sofic \cite{Ehyper}.

The embeddings $\phi$ and $\psi$ are {\it conjugate}
if there exists $\underline{g}\in\Sigma$ such that
$\phi(\gamma)=g\psi(\gamma)g^{-1}$ for any $\gamma\in\Gamma$.
We prove the following theorem.
\begin{theorem}\label{tetel2}
A finitely generated group $\Gamma$ is amenable if and only if any two
faithful embeddings of $\Gamma$ into $\Sigma$ are conjugate.
\end{theorem}
\vskip 0.1in
\noindent
{\bf Note added}: Liviu Paunescu informed us that he also has a proof of
Theorem 1. using a different argument. \cite{Pau}
\section{ Sofic approximations of amenable groups}
\begin{defin}
  Let $E(A)$ denote the set of edges in a colored graph $A$.
  We say that two colored graphs $A$ and $B$ are \emph{$r$-isomorphic}
  for some $r>0$
  if there are subgraphs $A'\subseteq A$ and $B'\subseteq B$
  such that 
  $$
  |E(A')|\ge\left(1-\textstyle\frac1r\right)|E(A)| \;,\quad
  |E(B')|\ge\left(1-\textstyle\frac1r\right)|E(B)|
  $$
  and $A'$ isomorphic to $B'$ (as colored graphs). The isomorphism
  between $A'$ and $B'$
  is called an \emph{$r$-isomorphism}.
\end{defin}

\begin{lemma} \label{r-isomorphism-transitivity}
  Let $A$, $B$ and $C$ be colored graphs,
  $A$ is $2r$-isomorphic to $B$ and $B$ is $2r$-isomorphic to $C$.
  Then $A$ is $r$-isomorphic to $C$.
  If $A$ is a $2r$-approximation of a Cayley graph and
  every vertex of $A$ has degree at most $d$
  then there is a bound $R_{\rm appr}(r,d)$ such that
  every colored graph $R_{\rm appr}(r,d)$-isomorphic to $A$ is
  an $r$-approximation of the Cayley graph.
\end{lemma}
\begin{proof}
  Clear from the definition.
\end{proof}

\begin{defin} \label{linear-combination-def}
  Let $\CT=T_1,T_2,\dots T_m$ be a finite sequence of colored graphs.
  Their \emph{linear combination with coefficient vector 
    $\alpha=(\alpha_1,\alpha_2,\dots\alpha_m)\in\Lattice$},
  denoted by $\alpha\cdot\CT$, is the disjoint union of
  $\alpha_1$ copies of $T_1$, $\alpha_2$ copies of $T_2$, $\dots$
  and $\alpha_m$ copies of $T_m$.
  As a special case,
  if $m=1$ then we talk about the \emph{integer multiples} of $T_1$.
\end{defin}

\begin{defin} \label{norm-def}\label{sigma-def}
  For vectors 
  $0\neq\alpha\in\Space$
  we define the unit vector
  $
  \sigma(\alpha)=\frac\alpha{\norm\alpha}
  $
  where $\norm\alpha$ denote the (usual) length of $\alpha$.
\end{defin}

\begin{lemma} \label{nearby-vectors}
  Let $\CT=\{T_1,T_2,\dots T_m\}$ be a finite sequence of colored
  graphs, each having at least one edge.
  For each integer $r>0$ there is an integer $M(r)>0$ 
  (also depending on $\CT$)
  such that
  whenever $\alpha,\beta\in\Lattice$ are nonzero lattice vectors with
  $\norm\beta\ge M(r)\norm\alpha$
  and $\norm{\sigma(\alpha)-\sigma(\beta)}\le\frac{1}{M(r)}$
  then the graph $\beta\cdot\CT$ is $r$-isomorphic to an integer
  multiple of $\alpha\cdot\CT$.
\end{lemma}

\proof
  Let $E$ be the maximum number of edges in $T_i$ and
  let $t$ be the largest integer such that
  $\norm{t\alpha}\le\norm{\beta}$.
  Clearly $M(r)\le t$ and $\norm\alpha\ge1$, hence
  $$
  \norm{\beta-t\alpha}\le
  \norm{\alpha}+t\norm{\sigma(\beta)-\sigma(\alpha)} \le
  \norm\alpha + t{\textstyle\frac{1}{M(r)}} \le
  {\textstyle\frac{2}{M(r)}}\norm\alpha t \;,
  $$
  hence each coordinate of $\beta$ differs by at most
  $\frac{2}{M(r)}\norm\alpha t$
  from the corresponding coordinate of $t\alpha$.
  Therefore $\beta\cdot\CT$ can be obtained from $(t\alpha)\cdot\CT$
  by adding or deleting at most $\frac{2}{M(r)}\norm\alpha t$ copies
  of each $T_i$, and the total number of edges involved in this
  operation is at most $\frac{2mE}{M(r)}\norm\alpha t$.

  On the other hand, the sum of the coordinates of $t\alpha$
  is at least $\norm{t\alpha}$ (triangle inequality),
  hence the graph $t \alpha\cdot\CT$ has at least
  $\norm\alpha t$ edges, and 
  similarly, the graph $\beta\cdot\CT$ has at least
  $\norm\beta\ge\norm\alpha t$  edges.
  If we choose $M(r)\ge2mEr$ then we obtain that
  $\beta\cdot\CT$ is $r$-isomorphic to $(t\alpha)\cdot\CT$,
  which is indeed an integer multiple of $\alpha\cdot\CT$.
\qed

\begin{defin} \label{Folner-def}
  Let $\Gamma$ be a finitely generated group,
  $S\subseteq \Gamma$ a finite symmetric generating set.
  A \emph{F\o lner} sequence in the Cayley graph $Cay(\Gamma,S)$
  is a sequence $\CF=\{F_1,F_2,\dots\}$ of spanned subgraphs
  such that for all $r>0$ all but finitely many of the $F_n$ are
  $r$-approximations. The group $\Gamma$ is \emph{amenable}
  if $Cay(\Gamma,S)$ has a F\o lner sequence.
  Note, that this property depends only on the group,
  i.e. any other Cayley graph of $\Gamma$ has a F\o lner sequence.
  More generally, an arbitrary group is \emph{amenable} if
  each of its finitely generated subgroups is amenable.
\end{defin}

\begin{propo} \label{tiling}
  [Reformulation of {\cite[Theorem 2.]{Elekame}}]
  Let $\Gamma$ be a finitely generated amenable group,
  $S\subseteq \Gamma$ a finite symmetric generating set
  and $\CF$ a F\o lner sequence in $Cay(\Gamma,S)$.
  Then for each integer $r>0$ there is an integer $R_0(r)>0$
  (which depends also on $\CF$)
  and a finite subsequence $\CT\subset\CF$ such that
  each $R_0(r)$-approximation of $Cay(\Gamma,S)$
  is $r$-equivalent to some linear combination of $\CT$. 
\end{propo}

\begin{propo} \label{mono}
  Let $\Gamma$ be a finitely generated amenable group,
  $S\subseteq \Gamma$ a finite symmetric generating set
  and $\CF$ a F\o lner sequence in $Cay(\Gamma,S)$.
  Then for each integer $r>0$ there is an integer $R_{\rm amenable}(r)$
  (which depends also on $\CF$)
  and a finite subsequence
  $\CT\subset\CF$ of say $m$ members
  and a lattice vector  $0\neq\alpha\in\Lattice$
  with the following property.
  Each member of $\CT$ is an $2r$-approximation of $Cay(\Gamma,S)$,
  and each $R_{\rm amenable}(r)$-approximation of $Cay(\Gamma,S)$
  is $2r$-equivalent to 
  an integer multiple of the (fixed) linear combination
  $\alpha\cdot\CT$.
  
\end{propo}

\proof
  By discarding a few elements from $\CF$ we may assume that the
  remaining elements are all $2r$-approximations of $Cay(\Gamma,S)$.
  We apply Proposition~\ref{tiling} with parameter $8r$,
  we obtain a subsequence $\CT\subset\CF$
  of say $m$ members such that
  every $R_0(8r)$-approximation of $Cay(\Gamma,S)$ is $8r$-isomorphic to some
  linear combination of $\CT$.
  Now we discard those elements of $\CF$ that are not
  $R_0(8r)$-approximation of $Cay(\Gamma,S)$.
  Note, that $\CT$ is no
  longer a subsequence of the remaining $\CF$.

  For each member $F\in\CF$ we choose a lattice vector
  $0\neq\beta_F\in\Lattice$
  such that $\beta_F\cdot\CT$ is $8r$-isomorphic to $F$.
  Now $\big\{\sigma(\beta_F)\big|F\in\CF\}$ is a sequence of unit vectors,
  so it has an accumulation point $u\in\Space$.
  We apply Lemma~\ref{nearby-vectors} to $\CT$ with parameter $8r$
  and obtain the bound $M(8r)$.
  Let us fix a member $H\in\CF$ such that 
  $\norm{\sigma(\beta_H)-u}<\frac{1}{2M(8r)}$,
  and then replace $\CF$ by the subsequence of those $F\in\CF$ which satisfy
  $\norm{\sigma(\beta_F)-u}<\frac{1}{2M(8r)}$ and 
  $\norm{\beta_F}\geq M(8r)\norm{\sigma(\beta_H)}$.
  The remaining sequence $\CF$ is still a F\o lner sequence
  and for each $F\in\CF$ we have
  $\norm{\sigma(\beta_F)-\sigma(\beta_H)}<\frac{1}{M(8r)}$,
  hence $\beta_F\cdot\CT$ is $8r$-isomorphic to
  an integer multiple of $\beta_H\cdot\CT$
  by Lemma~\ref{nearby-vectors}.

  Applying again Lemma~\ref{tiling} with parameter $8r$
  we obtain new finite subsequence $\{Q_1,Q_2,\dots Q_n\}$ of our
  modified $\CF$.
  Let $X$ be any $R_0(8)$-approximation of $Cay(\Gamma,S)$.
  Then $X$ is $8r$-isomorphic to a linear combination of the $Q_i$,
  each $Q_i$ is $8r$-isomorphic to the corresponding
  $\beta_{Q_i}\cdot\CT$,
  and in turn each $\beta_{Q_i}\cdot\CT$ is
  $8r$-isomorphic to an integer multiple of $\beta_H\cdot\CT$.
  Applying Lemma~\ref{r-isomorphism-transitivity} twice we obtain that
  $X$ is $2r$-isomorphic to an integer multiple of $\beta_H\cdot\CT$.
  The Lemma holds with $R_{\rm amenable}(r)=R_0(8r)$. \qed

\section{Amalgamated products}

\begin{defin} \label{incidence-graph}
Let $\alpha$ and $\beta$ be partitions of a finite set $S$. 
The {\em incidence graph} of $\alpha$ and $\beta$ is a bipartite
graph, whose two sets of vertices consist of the classes of $\alpha$
and the classes of $\beta$, and the edges are the elements of $S$, 
each element connects its
$\alpha$-class with its $\beta$-class.
\end{defin}

\begin{propo} [{proved in \cite[Step 1]{OnSofic}}]
  \label{partition-construction}
  For each triple $(a,b,r)$ of integers
  there is a finite set $S$ with two partitions $\alpha$ and $\beta$
  on it such that
  each $\alpha$-class has $a$ elements, each $\beta$-class has
  $b$ elements, an $\alpha$-class can meet a $\beta$-class in at
  most one element, and in the incidence graph of $\alpha$ and $\beta$
  each simple cycle is longer than $2r$.
\end{propo}

\begin{corol} \label{partition-corollary}
  Let $\alpha$, $\beta$ and $S$ be as in
  Proposition~\ref{partition-construction}.
  Suppose that the vertices of a graph are in $S$ and the endpoints of each 
  edge are 
  either in the same $\alpha$-class
  or in the same $\beta$-class.
  Then every simple cycle of length less than $2r$ 
  is  either entirely in a single $\alpha$-class
  or entirely in a single $\beta$-class. 
\end{corol}
\begin{proof}
  Follows immediately from Definition~\ref{incidence-graph}
\end{proof}

\begin{defin}
  Let $A$ be a colored graph and $Z$ a subset of the colors used in
  $A$.
  Then $A\big|_Z$ denotes the subgraph obtained from $A$ by
  omitting all edges whose color does not belong to $Z$.
  Suppose that $C$ is another colored graph and $\phi$ is an
  $r$-isomorphism between $A\big|_Z$ and $C\big|_Z$.
  We build another graph called the
  \emph{enhancement of $C$ with $A$ along $\phi$},
  denoted by $A\stackrel\phi\Rightarrow C$.
  We start from $C$ and add new edges to it.
  Namely, for each edge $a\to b$ of $A$ whose color does not belong to
  $Z$ we add a new edge $\phi(a)\to\phi(b)$ of the same color,
  provided that $\phi$ is defined at the endpoints $a$ and $b$.
\end{defin}

\begin{lemma} \label{glueing-lemma}
  Let $\Gamma$ be a finitely generated group,$X$ be a finite
  symmetric generating system of $\Gamma$ and $Z\subseteq X$.
  Let $A$ be a $2r$-approximation of $Cay(\Gamma,X)$
  such that every vertex has degree at most $d$.
  Then there is a bound $R_{\rm enh}(r,d)$ with the following property.
  Suppose that $C$ is another colored graph wit color set $Y$
  such that $X\cap Y=Z$ and
  $\phi$ is an $R_{\rm enh}(r,d)$-isomorphism
  between $A\big|_Z$ and $C\big|_Z$.
  Then $(A\stackrel\phi\Rightarrow C)\big|_{X}$ is
  $r$-isomorphic to $A$ and
  $(A\stackrel\phi\Rightarrow C)\big|_{Y}$ is isomorphic to $C$.
\end{lemma}

\begin{proof}
  Clear from the definition.
\end{proof}

\begin{lemma} \label{amalgamated-product-finitely-generated-case}
  Let $G$ and $H$ be finitely generated sofic groups,
  $\Gamma\le G$ a finitely generated amenable subgroup
  and $\phi:\Gamma\hookrightarrow H$ an injective homomorphism.
  Then the amalgamated product $G*_\phi H$ is also sofic.
\end{lemma}

\begin{proof}
  For simplicity we identify both $G$ and $H$ with their canonical
  image in $G*_\phi H$, this will also identify both $\Gamma$ and
  $\phi(\Gamma)$ with $G\cap H$.
  Let $X\subseteq G$, $Y\subseteq H$ and $Z\subseteq\Gamma$
  be finite symmetric generating systems
  such that $Z=X\cap Y$.
  Let $d$ be the maximum of $|X|$ and $|Y|$.
  %Recall the function $R_{\rm amenable}(r)$ from Proposition~\ref{mono}.
  The following is clear from the definition:

  \begin{claim} \label{Heureka}
    For each $r$ there is a constant $R_{\rm prod}(r)$
    with the following property.
    Suppose we find a colored graph $C$ with colors from $X\cup Y$
    such that $C\big|_X$ is an $R_{\rm prod}(r)$-approximation of $Cay(G,X)$,
    $C\big|_Y$ is an $R_{\rm prod}(r)$-approximation of $Cay(H,Y)$,
    and every simple cycle of length at most $2r+1$ has either only
    $X$-colored edges or only $Y$-colored edges.
    Then $C$ is an $r$-approximation of the
    $Cay\big(G*_\phi H,X\cup Y\big)$.
  \end{claim}

  We choose any integer $r>0$ and set
  \begin{eqnarray*}
    r_1&=&R_{\rm prod}(r),\\
    r_2&=&R_{\rm appr}(r_1,d),\\
    r_3&=&R_{\rm enh}(r_2,d),\\
    r_4&=&R_{\rm amenable}(r_3).
  \end{eqnarray*}

  Let us start with $r_4$-approximations $A$ resp. $B$
  of the Cayley graphs
  $Cay(G,X)$ resp. $Cay(H,Y)$.
  Then $A\big|_Z$ and $B\big|_Z$ be are both
  $r_4$-approximations of $Cay(\Gamma,Z)$.
  We choose a F\o lner sequence in $Cay(\Gamma,Z)$.
  Then Proposition~\ref{mono} provides a colored graph $D$ (denoted
  there by $\alpha\cdot\CT$) such that both $A\big|_Z$ and $B\big|_Z$ are
  $r_3$-isomorphic to integer multiples $a\cdot D$ and $b\cdot D$
  (with appropriate integers $a,b>0$).

  Applying Proposition~\ref{partition-construction} for the triple
  $(a,b,r+1)$ we obtain 
  a finite set $S$ and partitions $\alpha$ and $\beta$ of $S$.
  Let $S_0$ denote the graph with no edge of vertex set $S$.
  Then the graph $C_0=S_0\times D$ is the union of $|S|$ disjoint
  copies of $D$, namely the subgraphs $\{s\}\times D$ for $s\in S$. 

  Now for each $\alpha$-class $\sigma\subset S$
  the subgraph spanned by the subset $\sigma\times D$ 
  is $r_3$-isomorphic to $A\big|_Z$,
  so we can enhance it with $A$ along this $r_3$-isomorphism.
  Repeating this enhancement for all $\alpha$-classes
  we obtain a new graph $C_1$ on the same vertex set $S\times D$.
  According to Lemma~\ref{glueing-lemma},
  this $C_1$ is $r_2$-isomorphic to $|S|/a\cdot A$,
  and $C_1\big|_Z$ is still isomorphic to $|S|\cdot D$.

  Next for each $\beta$-class $\rho\subset S$
  we consider the subgraph $C_1^\rho\subset C_1$ spanned by the subset
  $\rho\times D$.
  Then $C_1^\rho\big|Z$ is $r_3$-isomorphic to $B\big|_Z$,
  so we can enhance it with $B$.
  Repeating this enhancement for all $\beta$-classes
  we obtain a new graph $C_2$ on the same vertex set $S\times D$.
  According to Lemma~\ref{glueing-lemma},
  this $C_2\big|_X$ is $r_2$-isomorphic to $|S|/a\cdot A$
  and $C_2\big|_Y$ is $r_2$-isomorphic to $|S|/b\cdot B$.
 
  We know from Lemma~\ref{r-isomorphism-transitivity} that
  $C_2\big|_X$ is an $r_1$-approximation of $Cay(G,X)$
  and $C_2\big|_Y$ is an $r_1$-approximation of $Cay(G,Y)$.
 Let $\gamma$ be a simple cycle of length at most $2r+1$ of $C_2$.
  Then by Lemma~\ref{partition-corollary},
  $\gamma$
  lies entirely in a single
  $\alpha$-class or it lies entirely in a single $\beta$ class.
  According to Claim~\ref{Heureka}, 
  $C_2$ is an $r$-approximation of $Cay\big(G*_\phi H,X\cup Y\big)$.
  Hence the lemma follows.
\end{proof}

\proof (of the main theorem) Let $G,H$ be arbitrary sofic groups, $\Gamma$
be an amenable subgroup of $G$ and let $\phi:\Gamma\to H$ be an injective
homomorphism. Then the amalgamated product $G\star_\phi H$ is the direct limit
of amalgamated products $\{G_\alpha\star_\phi H_\alpha\}$ over amenable
subgroups $\Gamma_\alpha$, where $G_\alpha$, $H_\alpha$ and $\Gamma_\alpha$
are all finitely generated groups. 
Since the direct limits of sofic groups are sofic, 
$G\star_\phi H$ is a sofic group. \qed

\section{Hyperfiniteness}
The notion of hyperfinite graph classes was introduced in \cite{Elekcost}.
A sofic approximation $\{G_n\}^\infty_{n=1}$ is hyperfinite if for any
 $\epsilon>0$ there
exists
 $K_\epsilon>0$ such that for any $n\geq 1$ one can erase $\epsilon |E(G_n)|$
 edges of $G_n$ to obtain a graph $G'_n$ having components of size
 not greater than $K_\epsilon$.
\begin{propo}\label{hyp} Let $\{G_n\}^\infty_{n=1}$ be a sofic
approximation of the Cayley graph of a finitely generated group $\Gamma$. Then
$\{G_n\}^\infty_{n=1}$ is hyperfinite if and only if $\Gamma$ is amenable.
\end{propo}
\proof
Let $\Gamma$ be amenable then $\{G_n\}^\infty_{n=1}$ is hyperfinite by
Proposition \ref{mono}. So let us suppose that $\Gamma$ is non-amenable and
let $S$ be a symmetric generating system of $\Gamma$.
The isoperimetric constant of $Cay(\Gamma,S)$ is defined as
$$i(Cay(\Gamma,S)):=\inf_{A\subset\Gamma\,,\,|A|<\infty}\frac{|L(A)|}{|A|}\,,$$
where $L(A)$ is the set of edges having one endpoint in $A$ and an other
endpoint in $A^c$. It is well-known that
$i(Cay(\Gamma,S))=c>0$. Now let $\{G_n\}^\infty_{n=1}$ be a sofic
approximation of $Cay(\Gamma,S)$ and suppose that $\{G_n\}^\infty_{n=1}$
is hyperfinite.
Let $\{G'_n\}^\infty_{n=1}$ be a graph sequence having components of size at
most $K$ such that each $G'_n$ is obtained by erasing $\frac{c}{10}| E(G_n)|$
edges from $G_n$. Observe that at least half of the vertices of $G_n$ are
in such a component $A$ of $G'_n$ such that
$$\frac{|L_{G_n}(A)|}{|A|}<\frac{c}{10}\,.$$
Notice that the $K$-neighborhood of such vertices can not be isomorphic
to the $K$-neighborhood of the unit element in $Cay(\Gamma,S)$. This is in
contradiction with the assumption that $\{G_n\}^\infty_{n=1}$ is a sofic
approximation. \qed
\section{The proof of Theorem \ref{tetel2}}
Let $\psi,\phi:\Gamma\to\Sigma$ be conjugate faithful embeddings of the group $\Gamma$.
Let $[\psi],[\phi]:\Gamma\to\prod^\infty_{n=1} S_n$ be representatives
of the maps $\psi$ resp. $\phi$. This gives us for each $n\geq 1$ the
$S$-colored graphs $G_n$ and $H_n$ on the set $\{1,2,\dots,n\}$.

\noindent
By the conjugacy assumption, for any $r\geq 1$ there exists subsequences
$\{G_{n_k}\}^\infty_{k=1}, \{H_{n_k}\}^\infty_{k=1}$ such that
\begin{itemize}
\item For any $k\geq 1$, $G_{n_k}$ and $H_{n_k}$ are $r$-isomorphic to each
  other.
\item Both $G_{n_k}$ and $H_{n_k}$ are $r$-approximations of $Cay(\Gamma,S)$.
\end{itemize}
Now let $\{G_n\}^\infty_{n=1}$ be a sofic approximation of  $Cay(\Gamma,S)$
such that $|V(G_n)|=n$ and $\Gamma$ is nonamenable.
By Proposition \ref{hyp}, we can suppose that there exists $\epsilon>0$ and a function
$r:\N\to\N$ tending to infinity such that; if we erase $\epsilon |E(G_n)|$
edges from $G_n$, then at least one of the remaining components has size at
least $r(n)$. Now consider the sofic approximation $\{H_n\}^\infty_{n=1}$,
where $H_n$ consists of $[\frac{n}{r(n)}]$ disjoint copies of $G_{r_n}$ and
$n-[\frac{n}{r(n)}]$ single vertices.
Let $\psi$ resp. $\phi:\Gamma\to\Sigma$ be the embeddings associated to the
sofic approximations $\{G_n\}^\infty_{n=1}$ resp.  $\{H_n\}^\infty_{n=1}$.
If $\psi$ and $\phi$ are conjugate then by the previous arguments there exists
$n_k>0$ such that $G_{n_k}$ and $H_{n_k}$ are $[\frac{1}{2
  \epsilon}]$-isomorphic to each other.
Therefore, one can erase $\epsilon |E(G_{n_k})|$ edges from $G_{n_k}$ to
obtain a graph having components of size at most $r(n_k)$, leading to a
contradiction. Thus there exist $\phi,\psi:\Gamma\to\Sigma$ nonconjugate
faithful embeddings provided that $\Gamma$ in nonamenable. 

\noindent
Now let $\Gamma$ be amenable and $\psi,\phi:\Gamma\to\Sigma$ be embeddings of
$\Gamma$.
Let $\{G_n\}^\infty_{n=1}$ and $\{H_n\}^\infty_{n=1}$ be the associated $S$-graphs
(for representatives $[\psi]$ and $[\phi])$.  For $n\geq 1$, let $h(n)$ be the
largest integer for which
$G_n$ and $H_n$ are both $h(n)$-approximations of $Cay(\Gamma,S)$. Using the
definition of the ultraproduct it is easy to see that for any $t\geq 1$
$$A_t=\{n\,\mid\, h(n)\geq t\}\in\omega\,.$$
By Proposition \ref{mono}, there exists a sofic approximation 
$\{J_n\}^\infty_{n=1}$ such that
\begin{itemize}
\item $|V(J_n)|=n$.
\item $G_n$ and $J_n$ resp. $H_n$ and $J_n$ are $2r$-isomorphic, whenever
$h(n)$ is large enough.
\end{itemize}
Therefore $G_n$ and $H_n$ are $r$-isomorphic if $h(n)$ is large enough.
Now we construct elements $g\in\prod_\omega S_n$ such that
$g\psi g^{-1}=\phi$.
For each $n\geq 1$ let $l(n)$ be the largest integer such that $G_n$ and $H_n$
are $l(n)$-isomorphic. Let $g_n\in S_n$ be a bijection such that
$g_n (G'_n)=H'_n$, where $G'_n$ resp $H'_n$ are isomorphic subgraphs
obtained by erasing $\frac{1}{l_n}|E(G_n)|$ resp. $\frac{1}{l_n}|E(H_n)|$ 
edges from $G_n$ resp. $H_n$. Then obviously $g\psi g^{-1}=\phi$.\qed

\end{document}